\documentclass[12pt, reqno]{amsart}
\usepackage{amsmath, amsthm, amscd, amsfonts, amssymb, graphicx, color}
\usepackage[bookmarksnumbered, colorlinks, plainpages]{hyperref}
\hypersetup{colorlinks=true,linkcolor=red, anchorcolor=green, citecolor=cyan, urlcolor=red, filecolor=magenta, pdftoolbar=true}

\newtheorem{theorem}{Theorem}[section]
\newtheorem{lemma}[theorem]{Lemma}
\newtheorem{proposition}[theorem]{Proposition}
\newtheorem{corollary}[theorem]{Corollary}
\theoremstyle{definition}

\newtheorem{example}[theorem]{Example}

\theoremstyle{remark}
\newtheorem{remark}[theorem]{Remark}
\numberwithin{equation}{section}

\begin{document}

\title[Matrix Inequalities between $f(A)\sigma f(B)$ and $A\sigma B$]{Matrix Inequalities between $f(A)\sigma f(B)$ and $A\sigma B$}

\author[M. Devi]{Manisha Devi}
\address{M. Devi: Department of Mathematics and Computing, Dr. B. R. Ambedkar National Institute of Technology, Jalandhar 144008, Punjab, India.}
\email{devimanisha076@gmail.com}

\author[J. S. Aujla]{Jaspal Singh Aujla}
\address{J. S. Aujla: Department of Mathematics and Computing, Dr. B. R. Ambedkar National Institute of Technology, Jalandhar 144008, Punjab, India.}
\email{aujlajs@nitj.ac.in}

\author[M. Kian]{Mohsen Kian}
\address{M. Kian: Department of Mathematics, University of Bojnord, P. O. Box
1339, Bojnord 94531, Iran}
 \email{kian@ub.ac.ir }

\author[M. S. Moslehian]{Mohammad Sal Moslehian}
\address{M. S. Moslehian: Department of Pure Mathematics, Center of Excellence in Analysis on Algebraic Structures (CEAAS), Ferdowsi University of Mashhad, P. O. Box 1159, Mashhad 91775, Iran.}
\email{moslehian@um.ac.ir; moslehian@yahoo.com}


\renewcommand{\subjclassname}{\textup{2020} Mathematics Subject Classification}

\maketitle

\begin{abstract}
Let $A$ and $ B$ be $n\times n$ positive definite complex matrices, let $\sigma$ be a matrix mean, and let $f : [0,\infty)\to [0,\infty)$ be a differentiable convex function with $f(0)=0$. We prove that
$$f^{\prime}(0)(A \sigma B)\leq \frac{f(m)}{m}(A\sigma B)\leq f(A)\sigma f(B)\leq \frac{f(M)}{M}(A\sigma B)\leq f^{\prime}(M)(A\sigma B),$$
where $m$ represents the smallest eigenvalues of $A$ and $B$ and $M$ represents the largest eigenvalues of $A$ and $B$. If $f$ is differentiable and concave, then the reverse inequalities hold. We use our result to improve some known subadditivity inequalities involving unitarily invariant norms under certain mild conditions. In particular, if $f(x)/x$ is increasing, then
$$|||f(A)+f(B)|||\leq\frac{f(M)}{M} |||A+B|||\leq |||f(A+B)|||$$
holds for all $A$ and $B$ with $M\leq A+B$.
Furthermore, we apply our results to explore some related inequalities. As an application, we present a generalization of Minkowski's determinant inequality.
\end{abstract}

\section{Introduction}

Throughout the paper, a capital letter represents a matrix in the algebra $\mathbb{M}_{n}$ of all $n\times n$ matrices over the field of complex numbers unless otherwise stated. We denote the identity matrix by $I$. The set of all Hermitian matrices in $\mathbb{M}_{n}$  is denoted by $\mathbb{H}_{n}$. The sets of all positive semidefinite and positive definite matrices in $\mathbb{H}_{n}$ are denoted by $\mathbb{P}_{n}$ and $\mathbb{P}_{n}^{+}$, respectively. A norm $|||\cdot |||$ on $\mathbb{M}_{n}$ is called \emph{unitarily invariant} if $|||UAV|||=|||A|||$ for all $A \in$ $\mathbb{M}_{n}$
and for all unitaries $U, V \in$ $\mathbb{M}_{n}.$ The most basic unitarily invariant norms are the operator norm $\|A\|=\sigma_{1}(A)$, the Schatten $p$-norms $\|\cdot\|_p$ for $p\geq 1$, and the Ky Fan norms $\|\cdot\|_{(k)}$, for $k=1,2,\ldots,n$, defined by
$$\|A\|_{(k)}=\sum_{j=1}^{k}\sigma_{j}(A) ~\hspace{.5cm}(k=1,2,\ldots,n),$$
where $\sigma_{1}(A)\geq \sigma_{2}(A)\geq\cdots\geq\sigma_{n}(A)$ are the singular values of $A$. Let $\lambda_j(A)$ stand for the $j$th eigenvalue of $A$ in decreasing order.

In addition, we assume that all real functions are continuous. Let $f$ be a real-valued function defined on an interval $J$ and let $A\in \mathbb{H}_{n}$ have its spectrum in $J$. Then $f(A)$ is defined by the familiar functional calculus. The function $f$ is called \emph{matrix monotone of order $n$} if $A\geq B$ implies $f(A)\geq f(B)$ for $A, B\in \mathbb{H}_{n}$ with spectra in $J$. If $f$ is matrix monotone of order $n$ for all $n\in \mathbb{N}$, then we say that $f$ is \emph{matrix monotone}. The function $f$ is said to be \emph{matrix convex} if
$$f(tA+(1-t)B)\leq tf(A)+(1-t)f(B)$$
for all $A,B\in \mathbb{H}_{n}$ with spectra in $J$ and for all $0\leq t\leq1$. The function $f$ is called \emph{matrix concave} if $-f$ is matrix convex. In the class of nonnegative functions on $[0,\infty)$, matrix monotonicity and matrix concavity are equivalent; see \cite[Theorem V.2.5]{7.}.

 A binary operation $\sigma: \mathbb{P}_{n} \times \mathbb{P}_{n}\to\mathbb{P}_{n}$ is called a \emph{matrix mean} (see \cite{FMPS}) provided that

 (i) $A\leq C$ and $B\leq D$ imply $A\sigma B\leq C\sigma D$;

 (ii) $C^*(A\sigma B)C\leq(C^*AC)\sigma(C^*BC)$;

 (iii) $A_n\downarrow A$ and $B_n\downarrow B$ imply $(A_n\sigma B_n)\downarrow A\sigma B$, where $A_n\downarrow A$ means that $A_1\geq A_2\geq \cdots$ and $A_n\to A$;

 (iv) $I\sigma I=I$.

There is an affine order isomorphism between the class of matrix means, and the class of positive matrix monotone functions $h$ defined on $[0,\infty)$ with $h(1)=1$ via $h(t)I=I\sigma (tI)\,\,(t\geq 0)$. In addition, $A \sigma B = A^{\frac{1}{2}}h(A^{\frac{-1}{2}}BA^{\frac{-1}{2}})A^{\frac{1}{2}}$.
For example, the matrix means corresponding to the positive matrix monotone functions $(1-t)x+t$, $((1-t)x^{-1}+t)^{-1}$, and $x^{t}$ are, respectively, called the \emph{weighted matrix arithmetic} mean $A\nabla_t B=(1-t)A+tB$, the \emph{weighted matrix harmonic mean} $A!_tB=((1-t)A^{-1}+tB^{-1})^{-1}$, and the \emph{weighted matrix geometric mean} $A\#_tB=A^{\frac{1}{2}}\left(A^{\frac{-1}{2}}BA^{\frac{-1}{2}}\right)^{t}A^{\frac{1}{2}}$. In cases where $A$ and $B$ are not invertible, these means are defined by a standard continuity argument.

Several comparisons between $|||f(A)+f(B)|||$ and $|||f(A+B)|||$ have been explored, particularly the Rotfel'd inequality $\|f(|A+B|)\|\leq \|f(|A|)\|+\|f(|B|)\|$. This topic has been studied by Kosem \cite{10.} as well as the second author and Silva \cite{5.}, Bourin and Uchiyama \cite{BOU}, Audenaert and the second author \cite{AUD}, Zhang \cite{ZHANG}, the third author, the fourth author and Mi\'{c}i\'{c} \cite{MMK}, as well as Alrimavi, Hirzallah, and Kittaneh \cite{KIT}. In addition, matrix inequalities that seek a relationship between $f(A)+f(B)$ and $f(A+B)$ have been examined by Singh and Vasudeva \cite{SINGH} as well as the second author and Bourin \cite{1.}.

In the second section, for $A,B\in \mathbb{P}_{n}^+$ and a matrix mean $\sigma$, we prove that
$$f^{\prime}(0)(A \sigma B)\leq \frac{f(m)}{m}(A\sigma B)\leq f(A)\sigma f(B)\leq \frac{f(M)}{M}(A\sigma B)\leq f^{\prime}(M)(A\sigma B),$$
and
$$f^{\prime}(0)(A \sigma B)\leq \frac{f(m)}{m}(A\sigma B)\leq f(A\sigma B)\leq \frac{f(M)}{M}(A\sigma B)\leq f^{\prime}(M)(A\sigma B),$$
when $f:[0,\infty)\rightarrow[0,\infty)$ is a differentiable convex function with $f(0)=0$, and $m$ and $M$ are the smallest and largest eigenvalues of $A$ and $B$, respectively. The above inequalities are reversed if $f$ is concave.

Our results imply some improvements to well-known norm subadditivity inequalities. Moreover, we prove that if $f(x)/x$ is increasing, then the inequality
$$|||f(A)+f(B)|||\leq\frac{f(M)}{M} |||A+B|||\leq |||f(A+B)|||$$
holds for all unitarily invariant norms on $\mathbb{M}_n$, when $M\leq A+B$.

We show related inequalities and present several examples to illustrate our discussion. In the third section, we apply our results to derive some determinantal inequalities. Most of our findings are applicable when the matrices $A$ and $B$ are replaced with appropriate operators on a separable Hilbert space. Throughout the paper, we assume that $m$ and $M$ are real numbers with $0 < m \leq M$ unless otherwise stated.

\section{Main results}

We start our work with the following auxiliary lemma.

\begin{lemma}\cite[Theorem 1.4.2]{NIC}\label{lemma21}
	Let $f$ be a convex function on an interval $J$. Then $f$ has finite left and right derivatives at each interior point of $J$, and $x<y$ in the interior of $J$ implies
	$$f_{-}^{\prime}(x)\leq f_{+}^{\prime}(x)\leq \frac{f(y)-f(x)}{y-x}\leq f_{-}^{\prime}(y)\leq f_{+}^{\prime}(y).$$
	The above inequalities are reversed if $f$ is a concave function.
\end{lemma}

The following result is the essence of our first theorem.
\begin{lemma}\label{12}
	Let $f:[0,\infty)\to[0,\infty)$ be a continuous function and $[m,M]\subseteq [0,\infty)$. Let $\sigma$ be a matrix mean, and $A, B\in \mathbb{P}_{n}$ with spectra in $[m,M]$. If $f$ is convex, then
	{\small\begin{align*}
		(a(A-mI)+f(m)I) \sigma (a(B-mI)+ f(m)I) &\leq f(A)\sigma f(B)\\
&\leq (b(A-mI)+f(m)I) \sigma ( b(B-mI)+f(m)I),
	\end{align*}}
	where $a=f_{+}^{\prime}(m)$ and $f_{-}^{\prime}(M)=b.$ If $f$ is concave, then the inequality is reversed.
	\end{lemma}
\begin{proof}
	It follows from Lemma \ref{lemma21} that for each $x$ with $m< x\leq M,$ we have
	$$a=f_{+}^{\prime}(m)\leq \frac{f(x)-f(m)}{x-m}\leq f_{-}^{\prime}(x)\leq f_{-}^{\prime}(M)=b.$$ The last inequality follows from the fact that $f^{\prime}_-$ is increasing. Hence,
	$$a(x-m)+f(m)\leq f(x) \leq b(x-m)+f(m).$$
	For all positive matrices $A$ and $B$, whose spectra are contained in the interval $[m,M]$, functional calculus implies that
	$$a(A-mI)+f(m)I\leq f(A) \leq b(A-mI)+f(m)I$$
	and $$a(B-mI)+f(m)I\leq f(B) \leq b(B-mI)+f(m)I.$$
	It follows from property (i) of matrix means that
	{\small\begin{align*}
		(a(A-mI)+f(m)I) \sigma (a(B-mI)+ f(m)I)& \leq f(A)\sigma f(B)\\
&\leq (b(A-mI)+f(m)I) \sigma ( b(B-mI)+f(m)I).
	\end{align*}}
	If $f$ is concave, the reverse inequality can be proved in the same fashion.
\end{proof}

%%-----------------------------------------------------------------------------------------------
\bigskip

From the above lemmas and property (ii) of matrix means, we conclude our first result.

\begin{theorem}\label{2.1}
	Let $f$ be a nonnegative differentiable convex function on $[0,\infty)$ with $f(0)=0$ and let $m,M$ be positive scalars with $m<M$. If $\sigma$ is a matrix mean, then
\begin{equation}\label{(2.1)}
f^{\prime}(0)(A \sigma B)\leq \frac{f(m)}{m}(A\sigma B)\leq f(A)\sigma f(B)\leq \frac{f(M)}{M}(A\sigma B)\leq f^{\prime}(M)(A\sigma B),
\end{equation}
and
\begin{equation}\label{(2.9)}
f^{\prime}(0)(A \sigma B)\leq \frac{f(m)}{m}(A\sigma B)\leq f(A\sigma B)\leq \frac{f(M)}{M}(A\sigma B)\leq f^{\prime}(M)(A\sigma B)
\end{equation}
hold for all $A,B\in \mathbb{P}_{n}^+$, whose spectra are contained in $[m,M]$. If the function $f$ is concave, then the inequalities are reversed.
\end{theorem}
\begin{proof}
	For $m\leq x\leq M,$ the convexity of $f$ yields
\begin{align}\label{k11}
f(x)=f\left(\frac{x}{M}M\right)\leq \frac{x}{M}f(M).
\end{align}
		This ensures that
		$$f(A)\leq \frac{f(M)}{M}A\qquad \mbox{and}\qquad f(B)\leq \frac{f(M)}{M}B $$
hold for all $A,B\in \mathbb{P}_{n}^+$ with spectra contained in $[m,M]$.
	By properties \rm{(i)} and \rm{(ii)} of matrix means, we get
\begin{align}\label{k1}
f(A)\sigma f(B)\leq \frac{f(M)}{M}(A\sigma B).
\end{align}
Again, it follows from the convexity of $f$ that
\begin{align}\label{k22}
f(x)=f\left(\frac{x}{m}m\right)\geq \frac{x}{m}f(m).
\end{align}
As mentioned above, this ensures that
\begin{align}\label{k2}
f(A)\sigma f(B)\geq \frac{f(m)}{m}(A\sigma B).
\end{align}
Utilizing Lemma \ref{lemma21} for the differentiable convex function $f$ we have
$$f^{\prime}(0)x\leq \frac{f(m)}{m}x \qquad \mbox{and}\qquad \frac{f(M)}{M}x\leq f^{\prime}(M)x,\qquad (0<m\leq x\leq M).$$
Consequently,
\begin{align}\label{k3}
f^{\prime}(0)(A\sigma B)\leq \frac{f(m)}{m}(A\sigma B)\qquad \mbox{and}\qquad \frac{f(M)}{M}(A\sigma B)\leq f^{\prime}(M)(A\sigma B).
\end{align}
It follows from \eqref{k1}, \eqref{k2} and \eqref{k3} that
$$f^{\prime}(0)(A \sigma B)\leq \frac{f(m)}{m}(A\sigma B)\leq f(A)\sigma f(B)\leq \frac{f(M)}{M}(A\sigma B)\leq f^{\prime}(M)(A\sigma B).$$
This proves \eqref{(2.1)}. To prove \eqref{(2.9)}, note that the spectrum of $A\sigma B$ is contained in $[m,M]$. Applying functional calculus with $x=A\sigma B$ in \eqref{k11} and \eqref{k22} we obtain
\begin{align*}
f(A\sigma B) \leq \frac{f(M)}{M}A\sigma B \qquad \mbox{and}\qquad f(A\sigma B)\geq \frac{f(m)}{m}A\sigma B.
\end{align*}
Now combine the above inequalities with \eqref{k3} to derive \eqref{(2.9)}.
\end{proof}
%%-------------------------------------------------------------------------------

\bigskip

\begin{example}
	For all $A,B\in \mathbb{P}_{n}$ with spectra in $[m,M]$, we have
	$$\frac{\log(M+1)}{M}\log(A+B+I)\leq \log(A+I)+\log(B+I).$$
To see this, take $f(x)=\log(x+1)$, which is a nonnegative differentiable concave function on $[0,\infty)$. If $\sigma=\nabla_{\frac{1}{2}}$, then making use of Theorem \ref{2.1} we obtain
$$\frac{\log(M+1)}{M}(A+B)\leq\log(A+I)+\log(B+I).$$
Since $\log(x+1)\leq x$ for every $x\geq0$, we have
$$\frac{\log(M+1)}{M}\log(A+B+I)\leq\frac{\log(M+1)}{M}(A+B).$$
This implies the desired result.
\end{example}
\begin{corollary}
Let $f$ be a nonnegative differentiable convex function on $[0,\infty)$ with $f(0)=0$. If $\sigma$ is a matrix mean, then
$$|||f(A)\sigma f(B)-f(A\sigma B)|||\leq (f^{\prime}(M)-f^{\prime}(0))|||A\sigma B||| $$
holds for all $A,B\in \mathbb{P}_{n}^+$ with spectra contained in $[m,M]$.
\end{corollary}

\bigskip

\begin{proof}
It follows from Theorem \ref{2.1} that
\begin{equation}\label{AM}
\begin{split}
f^{\prime}(0)(A \sigma B)&\leq f(A)\sigma f(B)\leq f^{\prime}(M)(A\sigma B),\\
f^{\prime}(0)(A \sigma B)&\leq f(A\sigma B)\leq f^{\prime}(M)(A\sigma B).
\end{split}
\end{equation}
Hence
$$(f^{\prime}(0)-f^{\prime}(M))(A\sigma B)\leq f(A)\sigma f(B)-f(A\sigma B)\leq(f^{\prime}(M)-f^{\prime}(0))(A\sigma B).$$
It is known that if $X\in \mathbb{H}_n$ and $Y\in \mathbb{P}_n$ satisfy $X<Y$ and $-X<Y$, then $|||X|||\leq|||Y|||$ for every unitarily invariant norm. Using this property, we get
$$|||f(A)\sigma f(B)-f(A\sigma B)|||\leq (f^{\prime}(M)-f^{\prime}(0))|||A\sigma B|||.$$
\end{proof}
%%---------------------------------------------------------------------
The next result is a direct consequence of Theorem \ref{2.1}.
\begin{corollary}\label{2.4}
Let $A$ and $B$ be positive definite matrices with spectra in $[m,M]$ and let $\sigma$ be a matrix mean. 	
If $f$ is a nonnegative differentiable convex function on $[0,\infty)$ with $f(0)=0$, then
	\begin{itemize}
		\item[(i)] For every $j=1,\dots,n$, the eigenvalue inequalities
{\small\begin{align*}
f^{\prime}(0) \lambda_{j}(A\sigma B)\leq \frac{f(m)}{m}\lambda_{j}(A\sigma B)\leq\lambda_{j}(f(A)\sigma f(B))\leq \frac{f(M)}{M}\lambda_{j}(A\sigma B)\leq f^{\prime}(M) \lambda_{j}(A\sigma B)
\end{align*}}
hold.
		\item[(ii)] For every $k=1,\dots,n$, the inequalities
{\small\begin{align*}
\prod_{j=1}^{k}f^{\prime}(0)\lambda_{j}(A\sigma B)& \leq\prod_{j=1}^{k}\frac{f(m)}{m}\lambda_{j}(A\sigma B)\\
&\leq\prod_{j=1}^{k}\lambda_{j}(f(A)\sigma f(B))\leq\prod_{j=1}^{k}\frac{f(M)}{M}\lambda_{j}(A\sigma B)\leq\prod_{j=1}^{k}f^{\prime}(M)\lambda_{j}(A\sigma B)
\end{align*}}
hold provided that $\lambda_j(A\sigma B)>0$.
		\item[(iii)] For every unitarily invariant norm $|||\cdot|||$ on $\mathbb{M}_n$, the inequalities
 {\small\begin{align*}
 f^{\prime}(0)|||A\sigma B|||\leq\frac{f(m)}{m}|||A\sigma B|||\leq |||f(A)\sigma f(B)|||\leq\frac{f(M)}{M}|||A\sigma B|||\leq f^{\prime}(M)|||A\sigma B|||,
 \end{align*}}
 hold.
	\end{itemize}
 If $f$ is a nonnegative differentiable concave function, the above inequalities are reversed.
\end{corollary}

\bigskip

\begin{remark}
Let $|||\cdot|||$ be a unitarily invariant norm on $\mathbb{M}_n$. Considering the matrix arithmetic mean, it follows from Corollary \ref{2.4} (iii) that the inequality
{\small\begin{align}\label{k-w}
\frac{f(m)}{m}|||A+ B|||\leq |||f(A)+ f(B)|||\leq\frac{f(M)}{M}|||A+ B|||
\end{align}}
holds for any differentiable convex function $f:[0,\infty)\to[0,\infty)$ with $f(0)=0$. If $f$ is concave, a reverse inequality holds. Bourin and Uchiyama \cite[Theorem 1.2]{BOU} (see also \cite{10.}) showed that
\begin{align}\label{bou}
|||f(A)+f(B)|||\leq |||f(A+B)|||
\end{align}
 is valid for every convex function $f:[0,\infty)\to[0,\infty)$ with $f(0)=0$ and all $A,B\in\mathbb{P}_n$. Now if $M\leq A+B$ and $f(x)/x$, $x\neq 0$ is increasing (for example, take $f(x)=x^r$, when $r\geq1$), then
 $$\frac{f(M)}{M}\leq \frac{|||f(A+B)|||}{|||A+ B|||}.$$
Hence, inequality \eqref{k-w} provides a better estimation than \eqref{bou}.
\end{remark}
%%--------------------------------------------------------------------------

\bigskip

\begin{remark}
The norm inequalities in Corollary \ref{2.4} (iii) do not hold for normal matrices, in general.
 For example, take $m=\min\{|\lambda_{j}(A)|,|\lambda_{j}(B)|\}$, $M=\max\{|\lambda_{j}(A)|,|\lambda_{j}(B)|\}$ and $f(x)=x^{2}$. With
	\begin{equation*}
		A =
		\begin{pmatrix}
			2 & 0 \\
			0 & -1
		\end{pmatrix}
		~\text{and} \hspace{0.3cm} B =
		\begin{pmatrix}
			-2 & 0 \\
			0 & 1
		\end{pmatrix}
	\end{equation*}
	we have
{\small\begin{align*}
\|f(|A|)+f(|B|)\|=8,\quad \|f(|A|+|B|)\|=16,\quad \frac{f(M)}{M}\|A+B\|=0,\quad \text{and}\,\, f^{\prime}(M)\|A+B\|=0.
\end{align*}}
	Therefore $$\|f(|A|)+f(|B|)\|\nleq\frac{f(M)}{M}\|A+B\|= f^{\prime}(M)\|A+B\| ~ \text{and} ~$$ $$ \|f(|A|+|B|)\|\nleq\frac{f(M)}{M}\|A+B\|= f^{\prime}(M)\|A+B\|.$$
\end{remark}

%%%------------------------------------------------------------------------------------------
However, we can generalize the left inequalities in Corollary \ref{2.4} (iii) to normal matrices. To achieve this goal, we need the following lemma.
\begin{lemma} \cite[Proposition 3.4]{BOU}\label{k-33}
	If $A,B\in \mathbb{M}_{n}$ are normal matrices, then
	$$|||A+B|||\leq |||~|A|+|B|~|||$$
for all unitarily invariant norms $|||\cdot|||$ on $\mathbb{M}_{n}$.
\end{lemma}

\begin{theorem}\label{thk}
	Let $A, B\in \mathbb{M}_{n}$ be normal matrices such that the spectra of $|A|$ and $|B|$ are contained in $[m,M]$. If $f$ is a nonnegative differentiable convex function on $[0,\infty)$ with $f(0)=0$, then
\begin{align}\label{kk-1}
f^{\prime}(0)|||A+B||| \leq\frac{f(m)}{m}||| A+B||| \leq|||~f(|A|)+f(|B|)~|||
\end{align}
	and
\begin{align}\label{kk-2}
f^{\prime}(0)||| A+B ||| \leq\frac{f(2m)}{2m}||| A+B ||| \leq||| f(|A|+|B|)|||.
\end{align}
	 If $f$ is a differentiable concave function, then
\begin{align*}
f^{\prime}(M)||| A+B ||| \leq\frac{f(M)}{M}||| A+B||| \leq |||f(|A|)+f(|B|)|||
\end{align*}
	and
\begin{align*}
f^{\prime}(M)||| A+B ||| \leq\frac{f(2M)}{2M}||| A+B||| \leq ||| f(|A|+|B|)|||.
\end{align*}
\end{theorem}
\begin{proof}
It follows from Lemma \ref{k-33} that
\begin{align}\label{kk-k1}
f^{\prime}(0)||| A+B ||| \leq \frac{f(m)}{m} ||| A+B ||| \leq \frac{f(m)}{m}||| ~|A|+|B|~|||.
\end{align}
Taking $\sigma=\nabla_{1/2}$ and 	applying Theorem \ref{2.1} to the positive matrices $|A|$ and $|B|$, we arrive at
\begin{align}\label{kk-k2}
\frac{f(m)}{m}||| ~|A|+|B|~||| \leq||| f(|A|)+f(|B|)|||.
\end{align}
Inequality \eqref{kk-1} now follows from \eqref{kk-k1} and \eqref{kk-k2}. For \eqref{kk-2}, use Lemma \ref{k-33} and Theorem \ref{2.1} to obtain
	\begin{align}
		f^{\prime}(0)\left|\left|\left|\frac{A+B}{2}\right|\right|\right|\leq \frac{f(m)}{m}\left|\left|\left|\frac{A+B}{2}\right|\right|\right|
\leq \frac{f(m)}{m}\left|\left|\left|\frac{|A|+|B|}{2}\right|\right|\right|
\leq\left|\left|\left|f\left(\frac{|A|+|B|}{2}\right)\right|\right|\right|.\nonumber
	\end{align}
	Replacing $A$ with $2A$ and $B$ with $2B$, we arrive at \eqref{kk-2}.
	Inequalities for concave functions can be similarly proved.
\end{proof}
%%------------------------------------------------------------------------------
\bigskip

\begin{remark}
\textbf{1.}\,  It has been shown in \cite[Theorem 2.1]{BOU2} that if $A, B\in \mathbb{M}_{n}$ are normal matrices and $f:[0,\infty)\to[0,\infty)$ is a concave function, then
 \begin{align}\label{bou2}
||| f(|A+B|)|||\leq ||| f(|A|)+f(|B|)|||
 \end{align}
 holds for all unitarily invariant norms $|||\cdot|||$ on $\mathbb{M}_{n}$. Let $A$ and $B$ be normal matrices such that the spectra of $|A|$ and $|B|$ are contained in $[m,M]$ and let $f:[0,\infty)\to[0,\infty)$ be a concave function. If $|A+B|\geq M$ and $f(x)/x$, $x\neq 0$ is decreasing, then
 $$\frac{||| f(|A+B|)|||}{||| A+B ||| }\leq \frac{f(M)}{M}.$$
Therefore, Theorem \ref{thk} gives a better estimation than \eqref{bou2}:
 $$||| f(|A+B|)|||\leq \frac{f(M)}{M}||| A+B||| \leq |||f(|A|)+f(|B|)|||.$$

\textbf{2.}\,   Let $\sigma=\sharp_\alpha$ with $\alpha\in[0,1]$ and $f(x)=x^r$, where $r\geq1$. Then Theorem \ref{2.1} ensures
   $$\lambda_{min}^{r-1}\, A\sharp_\alpha B \leq  A^r\sharp_\alpha B^r\leq \lambda_{max}^{r-1}\,A\sharp_\alpha B $$
   in which $\lambda_{min}=\min\{\lambda(A),\lambda(B)\}$ and $\lambda_{max}=\max\{\lambda(A),\lambda(B)\}$.

   If $\sigma$ is the matrix mean corresponding to the matrix monotone function $x\mapsto\log x$, then with
   $f(x)=x^r$, where $r\geq1$,  Theorem \ref{2.1} implies
   $$\lambda_{min}^{r-1}\, S(A|B) \leq  S(A^r|B^r)\leq \lambda_{max}^{r-1}\,S(A|B) $$
   in which $S(A|B)=A^{1/2}\log(A^{-1/2}BA^{-1/2})A^{1/2}$  is the relative operator entropy.

\textbf{3.}\, We remark that in the special case $\sigma=\sharp_\alpha$ with $\alpha\in[0,1]$ and $f(x)=x^r$, where $r\geq1$, the Ando--Hiai inequality \cite{AHI} (see also \cite{kianm} and the references therein) provides a sharp estimation. In fact, the Ando--Hiai inequality states that
\begin{align}\label{AH}
 A^r\sharp_\alpha B^r \leq \|A\sharp_\alpha B\|^{r-1} A\sharp_\alpha B\qquad (0\leq \alpha\leq1,\,\, r\geq1).
\end{align}
Now, if $A$ and $B$ are two positive definite matrices, $\|A\|\leq \|B\|$, $f(x)=x^r$ and $\sigma=\sharp_\alpha$, then Theorem \ref{2.1} gives
 $$A^r\sharp_\alpha B^r\leq \frac{f(\|B\|)}{\|B\|}A\sharp_\alpha B=\|B\|^{r-1}A\sharp_\alpha B.$$
Note that  $\|A\sharp_\alpha B\|^{r-1}\leq (\|A\|\sharp_\alpha \|B\|)^{r-1}\leq \|B\|^{r-1}$.

\end{remark}
%%--------------------------------------------------------------------
In this direction, we present the next version of our result.
 \begin{theorem}\label{thh}
 Let $A$ and $B$ be positive definite matrices. Let $g$ and $h$ be positive matrix monotone functions on $[0,\infty)$ satisfying
 \begin{align}\label{d}
 \begin{split}
 & \mathrm{(i)}\quad g\circ \frac{1}{h}\leq \frac{1}{h}\circ g, \\
 &\mathrm{(ii)}\quad g\circ \frac{x}{h}\leq \frac{x}{h}\circ g,\\
 &\mathrm{(iii)}\quad h(x g(x))\leq h(x) h(g(x)).
 \end{split}
 \end{align}
 Then
 $$A\sigma_h B\leq I \qquad \Longrightarrow\qquad f(A)^n\sigma_h f(B)^n\leq I$$
 holds, where $f(x)=xg(x)$ and $f^2=f\circ f$. If the inequalities in \eqref{d} are reversed, then
 $$A\sigma_h B\geq I \qquad \Longrightarrow\qquad f(A)^n\sigma_h f(B)^n\geq I.$$
 \end{theorem}
 Before proving the theorem, we give some examples to demonstrate that the conditions on $g$ and $h$ mentioned above are not overly restrictive. If $p,q\in[0,1]$, then $g_1(x)=x^p$ and $h_1(x)=x^q$ satisfy \eqref{d}. The same holds for the pair of matrix monotone functions $(g_2,h_2)$, where $g_2(x)=\log(x)$ and $h_2(x)=x/\log(x)$. Moreover, certain other matrix monotone functions meet these conditions when their domains are restricted. Hence, the theorem is valid when the spectra of $A$ and $B$ are limited.

 Furthermore, it is known that a continuous function
 $f$ on $[0,\infty)$ is matrix convex and $f(0)\leq0$ if and only if $x\mapsto f(x)/x$ is matrix monotone.
 Hence, Theorem \ref{thh} ensures the next result.
 \begin{corollary}\label{thh2}
 Let $A$ and $B$ be positive definite matrices. Let $f$ be a matrix convex function, $h$ be a matrix monotone function, and $f(0)\leq0$. If
 \begin{align}\label{dd}
 \begin{split}
 & \mathrm{(i)}\quad \frac{f(x)}{x}\circ \frac{1}{h}\leq \frac{1}{h}\circ \frac{f(x)}{x}, \\
 &\mathrm{(ii)}\quad \frac{f(x)}{x}\circ \frac{x}{h}\leq \frac{x}{h}\circ \frac{f(x)}{x},\\
 &\mathrm{(iii)}\quad h(f(x))\leq h(x) h\left(\frac{f(x)}{x}\right),
 \end{split}
 \end{align}
then
 $$A\sigma_h B\leq I \qquad \Longrightarrow\qquad f(A)^n\sigma_h f(B)^n\leq I.$$
 \end{corollary}
 \begin{proof} ({\it of Theorem \ref{thh}}).
 Suppose that $g(x)$ and $h(x)$ are matrix monotone functions and $A,B\in\mathbb{P}_n^+$.
 Set $k(x):=xg(x^{-1})$. Since $g$ is matrix monotone, the function $x\mapsto x/g(x)$ is also matrix monotone as shown in \cite[Corollary 1.14]{FMPS}. Therefore, the function $x\mapsto g(x)/x$ is a matrix monotone decreasing function. This implies that $k(x)=xg(x^{-1})$ is also matrix monotone. Let $\sigma_k$ be the matrix mean corresponding to the function $k$. Put $X=A^{-\frac{1}{2}}BA^{-\frac{1}{2}}$ and $Y=A^{-\frac{1}{2}}Bg(B) A^{-\frac{1}{2}}$. Then $B=A^\frac{1}{2}XA^\frac{1}{2}$ and we have
 \begin{align*}
 Y=A^{-\frac{1}{2}}Bg(B) A^{-\frac{1}{2}}& = A^{-\frac{1}{2}}B(B^{-1}g(B))B A^{-\frac{1}{2}}\\
 &=XA^\frac{1}{2}(B^{-1}g(B))A^\frac{1}{2}X \qquad(\mbox{since $A^{-\frac{1}{2}}B=XA^\frac{1}{2}$}) \\
 &=XA^\frac{1}{2}k(B^{-1})A^\frac{1}{2}X\\
 &=XA^\frac{1}{2} k\left(A^{-\frac{1}{2}} X^{-1} A^{-\frac{1}{2}} \right)A^\frac{1}{2}X\\
 &= X\left(A\sigma_k X^{-1}\right)X.
 \end{align*}
 Since $A\sigma_h B=A^\frac{1}{2}h(X)A^\frac{1}{2}$, it follows that
 $$A\sigma_h B\leq I \qquad \Leftrightarrow\qquad h(X)\leq A^{-1}\qquad \Leftrightarrow\qquad A\leq h(X)^{-1}.$$
Therefore, from the monotonicity of matrix means, we infer that
 \begin{align}\label{mk2}
 \begin{split}
 Y&=X\left(A\sigma_k X^{-1}\right)X\\
 &\leq X\left(h(X)^{-1}\sigma_k X^{-1}\right)X\qquad(\mbox{by $A\leq h(X)^{-1}$})\\
 & = X^2h(X)^{-1} k \left(h(X)X^{-1}\right)\qquad(\mbox{since $X$ and $h(X)$ commute})\\
 &= X g\left(h(X)^{-1}X\right)\qquad\qquad(\mbox{by $k(x)=xg(x^{-1})$})\\
 &\leq X g(X) h(g(X))^{-1},
 \end{split}
 \end{align}
where the last inequality follows from hypothesis {\rm (ii)} in \eqref{d}. Furthermore, the matrix monotonicity of $g$ together with hypothesis {\rm (i)} in \eqref{d} gives
\begin{align}\label{df}
g(A)\leq g\left(h(X)^{-1}\right)\leq h(g(X))^{-1}.
\end{align}
Therefore,
\begin{align}\label{edf}
\begin{split}
g(A)\sigma_h Y &\leq h(g(X))^{-1} \sigma_h X g(X) h(g(X))^{-1}\qquad(\mbox{by \eqref{mk2} and \eqref{df}})\\
&= h(g(X))^{-1} h\left(X g(X)\right) \\
&\leq h(g(X))^{-1} h(X) h(g(X))=h(X),
 \end{split}
\end{align}
in which the last inequality comes from part {\rm (iii)} in \eqref{d}.
Now put $f(x)=xg(x)$. Since $A$ and $g(A)$ commute, we can write
 \begin{align}\label{mk1}
 \begin{split}
 f(A)\sigma_h f(B)&=(Ag(A))^\frac{1}{2}h\left(g(A)^{-\frac{1}{2}}A^{-\frac{1}{2}}Bg(B)
 A^{-\frac{1}{2}}g(A)^{-\frac{1}{2}}\right) (Ag(A))^\frac{1}{2}\\
 &=A^\frac{1}{2}\left(g(A)\sigma_h \left(A^{-\frac{1}{2}}Bg(B)
 A^{-\frac{1}{2}}\right)\right)A^\frac{1}{2}\\
 &=A^\frac{1}{2}\left(g(A)\sigma_h Y\right)A^\frac{1}{2}\\
 &\leq A^\frac{1}{2}h(X) A^\frac{1}{2}\qquad\qquad(\mbox{by \eqref{edf}})\\
 &=A\sigma_h B\leq I,
 \end{split}
 \end{align}
 as required. Accordingly, $f^n(A) \sigma_h f^n(B)\leq A\sigma_h B\leq I$.
 \end{proof}

%%----------------------------------------------------------------------------------

%%---------------------------------------------------------------------------------

 \bigskip

One may observe that a function can be neither convex nor concave, but its inverse function (if it
exists) may be convex or concave. Examples of such functions include $\tan x$ and
$\cot x$. The function $f(x)=x+\frac{1}{x}$ on $(0, \infty)$ does not satisfy inequality \eqref{(2.1)}, but its inverse function $f^{-1}(x)=\frac{1}{2}(x+\sqrt{x^{2}-4})$ is concave on the interval $[2,\infty)$. Therefore, it makes sense to consider the following result.
\begin{proposition}
	Let $f$ be a nonnegative increasing function on $[0,\infty)$ with $f(0)=0$. If the inverse function $f^{-1}$ is differentiable and convex, then
	$$\frac{f(M)}{M}(A\sigma B)\leq f(A)\sigma f(B)\leq \frac{f(m)}{m}(A\sigma B)$$
for all $A,B\in \mathbb{P}_{n}^+$, whose spectra are contained in $[m,M]$. If $f^{-1}$ is differentiable and concave, then
$$\frac{f(m)}{m}(A\sigma B)\leq f(A)\sigma f(B)\leq \frac{f(M)}{M}(A\sigma B).$$
\end{proposition}
\begin{proof}
	Since $f^{-1}$ satisfies the conditions of Theorem \ref{2.1}, we have
	$$\frac{(f^{-1})(m)}{m}(A\sigma B)\leq f^{-1}(A)\sigma f^{-1}(B)\leq\frac{(f^{-1})(M)}{M}(A\sigma B).$$
	Since $m\leq A,B\leq M$ and $f$ is increasing, $f(A)$ and $f(B)$ are positive matrices, whose spectra are in $[f(m),f(M)]$. Hence, we can replace $A$ and $B$ with $f(A)$ and $f(B)$, respectively to obtain
	$$\frac{(f^{-1})(f(m))}{f(m)}(f(A)\sigma f(B))\leq A\sigma B\leq\frac{(f^{-1})(f(M))}{f(M)}(f(A)\sigma f(B)),$$
which is equivalent to
	$$\frac{f(M)}{M}(A\sigma B)\leq f(A)\sigma f(B)\leq \frac{f(m)}{m}(A\sigma B).$$
The case of concave functions is similarly proved.
\end{proof}

\section{Applications to Determinantal inequalities}
The famous Minkowski's determinant inequality \cite[Corollary II.3.21]{7.} states that if $A, B\in\mathbb{P}_{n}$, then
\begin{equation}\label{1}
	(\det A)^{\frac{1}{n}}+(\det B)^{\frac{1}{n}}\leq (\det(A+B))^{\frac{1}{n}}.
\end{equation}
In this section, utilizing our result in Section 2, we generalize the above inequality. We need the following lemma.
\begin{lemma}\cite[Theorem 6]{MIC}\label{lemma22}
	Let $A, B\in \mathbb{P}_{n}^{+}$ and let $\alpha, \beta>0$ with $\alpha+\beta=1$. Then
	\begin{equation}\label{4}
		\det(\alpha A+\beta B)\leq \alpha\det A+\beta\det B
	\end{equation}
if any one of the following conditions is satisfied:
	\begin{itemize}
		\item[(i)] $\lambda_{j}(B)<\lambda_{n}(A)$ for $j=1,2,\ldots,n,$
		\item[(ii)] $\lambda_{1}(A)<\lambda_{j}(B)$ for $j=1,2,\ldots,n.$
	\end{itemize}
\end{lemma}

\begin{theorem}\label{5}
	Let $f$ be a nonnegative differentiable convex function on $[0,\infty)$ with $f(0)=0$. If $A, B\in \mathbb{P}_{n}^+$ and $0<m\leq A,B\leq M$, then
	\begin{itemize}
		\item[(i)] $(\det f(A))^{\frac{1}{n}}+(\det f(B))^{\frac{1}{n}}\leq \frac{f(M)}{M}(\det(A+B))^{\frac{1}{n}},$
		\item[(ii)] $\frac{f(m)}{m}((\det A)^{\frac{1}{n}}+(\det B)^{\frac{1}{n}})\leq (\det(f(A)+f(B)))^{\frac{1}{n}}.$
	\end{itemize}
	If the function $f$ is differentiable concave, then
	\begin{itemize}
		\item[(iii)] $(\det f(A))^{\frac{1}{n}}+(\det f(B))^{\frac{1}{n}}\leq \frac{f(m)}{m}(\det(A+B))^{\frac{1}{n}},$
		\item[(iv)] $\frac{f(M)}{M}((\det A)^{\frac{1}{n}}+(\det B)^{\frac{1}{n}})\leq (\det(f(A)+f(B)))^{\frac{1}{n}}.$
	\end{itemize}
\end{theorem}
\begin{proof}
 Applying part (ii) of Corollary \ref{2.4} with $\sigma=\nabla_{\frac{1}{2}}$ and $k=n$, we obtain
	{\small\begin{align*}
		f^{\prime}(0)(\det(A+B))^{\frac{1}{n}}&\leq\frac{f(m)}{m}(\det(A+B))^{\frac{1}{n}}\\
&\leq (\det(f(A)+f(B)))^{\frac{1}{n}}\\
& \leq\frac{f(M)}{M}(\det(A+B))^{\frac{1}{n}}
		 \leq f^{\prime}(M)(\det(A+B))^{\frac{1}{n}}.
	\end{align*}}
In particular,
	 \begin{align}\label{po1}
(\det(f(A)+f(B)))^{\frac{1}{n}}\leq \frac{f(M)}{M}(\det(A+B))^{\frac{1}{n}}
	\end{align}
and
	 \begin{align}\label{qo1}
\frac{f(m)}{m}(\det(A+B))^{\frac{1}{n}}\leq (\det(f(A)+f(B)))^{\frac{1}{n}}.
	\end{align}
	Minkowski's determinant inequality \eqref{1} implies that
 \begin{align}\label{po2}
 (\det f(A))^{\frac{1}{n}}+(\det f(B))^{\frac{1}{n}}\leq(\det(f(A)+f(B)))^{\frac{1}{n}}.
 	\end{align}
	Part (i) now follows from \eqref{po1} and \eqref{po2}. Using the above two inequalities, we get our desired result. Part (ii) follows analogously by employing \eqref{qo1} and the
 Minkowski determinant inequality.
	The proofs of (iii) and (iv) for concave functions are similar.
\end{proof}

Taking $f(x)=x$, from Theorem \ref{5} we get the Minkowski determinant inequality.

\bigskip

In the next result, we present a reverse type inequality for the Minkowski determinant inequality.

\begin{theorem}\label{6}
Let $A, B\in \mathbb{P}_{n}^+$ and $0<m\leq A,B\leq M$ and let $f$ be a nonnegative differentiable convex function on $[0,\infty)$ with $f(0)=0$. Let $\alpha, \beta>0$ with $\alpha+\beta=1$. If one of the following conditions is satisfied:
	\begin{itemize}
		\item[(i)] $\lambda_{j}(B)<\lambda_{n}(A)$ for $j=1,2,\ldots,n,$
		\item[(ii)] $\lambda_{1}(A)<\lambda_{j}(B)$ for $j=1,2,\ldots,n,$
	\end{itemize}
	then
	\begin{equation*}
		(\det(f(A)+f(B)))^{\frac{1}{n}}\leq 2^{1-\frac{1}{n}}\frac{f(M)}{M}((\det A)^{\frac{1}{n}}+(\det B)^{\frac{1}{n}}).
	\end{equation*}
\end{theorem}
\begin{proof}
 Using inequality \eqref{po1} and Lemma \ref{lemma22}, we can write
	\begin{align*}
		(\det(f(A)+f(B)))^{\frac{1}{n}}&\leq \frac{f(M)}{M}\left(\det\left(\frac{2A+2B}{2}\right)\right)^{\frac{1}{n}}\qquad(\mbox{by \eqref{po1}}) \\
		&\leq \frac{f(M)}{M}\left(\frac{\det(2A)+\det(2B)}{2}\right)^{\frac{1}{n}}\qquad(\mbox{by Lemma \ref{lemma22}}) \\
		&=\frac{f(M)}{M}\left(\frac{2^{n}}{2}(\det A+\det B)\right)^{\frac{1}{n}} \\
		&=\frac{f(M)}{M}2^{\frac{n-1}{n}}(\det A+\det B)^{\frac{1}{n}} \\
		&\leq2^{1-\frac{1}{n}}\frac{f(M)}{M}\left((\det A)^{\frac{1}{n}}+(\det B)^{\frac{1}{n}}\right) ,
	\end{align*}
	where the last inequality follows from the subadditivity of the function $x^{\frac{1}{n}}.$ This concludes the desired inequality.
\end{proof}

	\begin{remark}
		Taking $f(x)=x$ in Theorem \ref{6}, we have
		$$(\det(A+B))^{\frac{1}{n}}\leq 2^{1-\frac{1}{n}}\left((\det A)^{\frac{1}{n}}+(\det B)^{\frac{1}{n}}\right),$$
which is a reverse type inequality for Minkowski's determinant inequality.
	\end{remark}

\medskip
\textbf{Acknowledgments:}
The authors are thankful to Professor Mandeep Singh for making valuable suggestions.

\medskip

\noindent \textit{Conflict of Interest Statement.} On behalf of all authors, the corresponding author states that there is no conflict of interest.\\

\noindent\textit{Data Availability Statement.} Data sharing not applicable to this article as no datasets were generated or analysed during the current study.

\medskip

\bibliographystyle{amsplain}

\end{document}